\documentclass[11pt,reqno]{amsart}

\usepackage{amssymb, amsmath}
\usepackage{hyperref}
\usepackage[nohug]{diagrams}
\diagramstyle[labelstyle=\scriptstyle]

\newtheorem{proposition}{Proposition}[section]
\newtheorem{lemma}[proposition]{Lemma}
\newtheorem{corollary}[proposition]{Corollary}
\newtheorem{theorem}[proposition]{Theorem}

\theoremstyle{definition}
\newtheorem{definition}[proposition]{Definition}

\newtheorem{remark}[proposition]{Remark}

\newenvironment{prop}[2][Proposition]{\begin{trivlist}
\item[\hskip \labelsep {\bfseries #1}\hskip \labelsep {\bfseries #2}]}{\end{trivlist}}

\newcommand{\thlabel}[1]{\label{th:#1}}
\newcommand{\thref}[1]{Theorem~\ref{th:#1}}
\newcommand{\selabel}[1]{\label{se:#1}}
\newcommand{\seref}[1]{Section~\ref{se:#1}}
\newcommand{\lelabel}[1]{\label{le:#1}}
\newcommand{\leref}[1]{Lemma~\ref{le:#1}}
\newcommand{\prlabel}[1]{\label{pr:#1}}
\newcommand{\prref}[1]{Proposition~\ref{pr:#1}}
\newcommand{\colabel}[1]{\label{co:#1}}
\newcommand{\coref}[1]{Corollary~\ref{co:#1}}
\newcommand{\relabel}[1]{\label{re:#1}}
\newcommand{\reref}[1]{Remark~\ref{re:#1}}

\newcommand{\delabel}[1]{\label{de:#1}}

\newcommand{\eqlabel}[1]{\label{eq:#1}}
\newcommand{\equref}[1]{(\ref{eq:#1})}

\numberwithin{equation}{section}

\newcommand{\mc}{\mathcal}
\newcommand{\mb}{\mathbb}
\newcommand{\mf}{\mathfrak}
\newcommand{\ds}{\displaystyle}

\newcommand{\ol}{\overline}
\newcommand{\la}{\langle}  \newcommand{\ra}{\rangle}

\newcommand{\analog}{\rule{1cm}{0.01mm}}

\begin{document}

\title{Grothendieck rings of universal quantum groups}
\author{Alexandru Chirv\u asitu}

\address{
University of California, Berkeley, 970 Evans Hall \#3480, Berkeley CA, 94720-3840, USA
}
\email{chirvasitua@gmail.com}

\subjclass[2010]{16T05, 16T15, 16T20, 20G42}
\keywords{free Hopf algebra, cosovereign Hopf algebra, matrix coalgebra, Grothendieck ring of comodules, corepresentation}

\begin{abstract}

We determine the Grothendieck ring of finite-dimensional comodules for the free Hopf algebra on a matrix coalgebra, and similarly for the free Hopf algebra with bijective antipode and other related universal quantum groups. The results turn out to be parallel to those for Wang and Van Daele's deformed universal compact quantum groups and Bichon's generalization of those results to universal cosovereign Hopf algebras: in all cases the rings are isomorphic to those of non-commutative polynomials over certain sets, these sets varying from case to case. In most cases we are able to give more precise information about the multiplication table of the Grothendieck ring.  

\end{abstract}

\maketitle

\section*{Introduction}\selabel{0}

The representation theory of quantum groups has played an important role in mathematics during the past several decades. Several approaches can be identified, which yield interesting different, but often related families of Hopf algebras. One has, for example, Drinfeld and Jimbo's deformed universal enveloping algebras (\cite{Dr1, Dr2, Ji}), the compact matrix groups of Woronowicz (\cite{Wo1, Wo2}), or various ``quantum automorphism groups'', such as those of Manin (\cite{Ma}), the quantum group of a bilinear form (\cite{DVL}), that of a measured algebra (\cite{Bi1}), etc. 

The ``universal quantum groups" in the title are Hopf algebras which enjoy certain universality properties; they are described in more detail below. We are interested in their finite-dimensional comodules, so they are to be regarded as quantum groups of the ``function algebra" flavor.  

One class of Hopf algebras which will be relevant to our discussion and will provide the motivation for what follows is that of universal or free {\it cosovereign} Hopf algebras. These were introduced by Bichon in \cite{Bi2}, and are defined essentially as follows: given an invertible $n\times n$ matrix $F$, the universal cosovereign Hopf algebra $H(F)$ is the free Hopf algebra generated by an $n\times n$ matrix coalgebra $u=(u_{ij})$ with the provision that the squared antipode acts on $u$ as conjugation by $F$ (see \cite{Bi2} for more details).

The main objects of study here are the following:

(1) The free Hopf algebra $H(n)$ on the matrix coalgebra $M_n(k)^*$ (for some field $k$ and $n\ge 2$). It was shown in \cite{Ta} that the forgetful functor from Hopf algebras to coalgebras (always over some fixed base field $k$) has a left adjoint. $H(n)$ is precisely the image of the matrix coalgebra $M_n(k)*$ through this adjoint. 

(2) $H_\infty(n)$, the free Hopf algebra with bijective antipode on the same matrix coalgebra $M_n(k)^*$. As in (1), it is shown in \cite{Sc} that the forgetful functor from Hopf algebras with bijective antipode to that of coalgebras has a left adjoint. Just as before, $H_\infty(n)$ denotes here the image of the matrix coalgebra through that adjoint.

(3) We introduce an object denoted by $H_d(F)$. Here $d$ is a positive integer, while $F$ is an invertible $n\times n$ matrix over $k$. With this data, $H_d(F)$ is the free Hopf algebra generated by a matrix coalgebra $u=(u_{ij})$ such that the $2d$'th power of the antipode acts on $u$ as conjugation by $F$. We chose to consider these objects because they generalize at the same time the universal cosovereign Hopf algebras discussed above ($H(F)$ from \cite{Bi2} would be $H_1(F)$ here), and the free Hopf algebra with antipode of order $2d$ on a matrix coalgebra, used in \cite{Ch} ($H_d(M_n(k)^*)$ from that paper is $H_d(I_n)$ here, where $I_n\in M_n(k)$ is the identity matrix). 
      
Finally, we reserve the notation $\tilde H$ or $\tilde H(n)$ as a placeholder for any of the above; the $n$ indicates that we are considering either $H(n)$, or $H_\infty(n)$, or $H_d(F)$ for some $n\times n$ matrix $F\in GL(n,k)$. We will be concerned primarily with determining the Grothendieck rings of finite-dimensional comodules for the various $\tilde H(n)$'s.

It turns out that when the base field is $\mb C$ and the matrix $F$ used in the definition of $H(F)$ is positive definite, the $H(F)$ are precisely the CQG algebras (in the sense of \cite{DK}, for example) associated to Wang and Van Daele's compact quantum matrix groups $A_u(Q)$ (\cite{VDW}). The corepresentations of the latter were determined by B\u anic\u a in \cite{Ba}, and the results were later generalized by Bichon (\cite{Bi4}) to include all cosemisimple $H(F)$'s in characteristic zero. The corepresentations of $A_u(Q)$ (and by extension those of $H(F)$) are of interest because collectively, the $A_u(Q)$ play the role of the unitary group $U(n)$ (see \cite{Ba}). We will recall the relevant results in the next section. 

This discussion provides part of the motivation for our problem: the combinatorics of the multiplication table for the Grothendieck rings under consideration turns out to mimic the results obtained in \cite{Ba} and \cite{Bi4} quite closely, and seems interesting in its own right. Essentially, our results say that at least for $\tilde H(n)$ excluding $H_1(F)$, the Grothendieck ring is ``as free of relations'' as one can expect (see the next section for precise statements).  

Further motivation comes from the desire to obtain more information on the free Hopf algebras $H(n)$ (and their relatives). Ever since the introduction of $H(n)$ (and in fact of the free Hopf algebra on any coalgebra) by Takeuchi in \cite{Ta}, where they were used to give the first examples of Hopf algebras with non-bijective antipode, they have appeared in several other papers, also as the basis for counterexamples: in \cite{Ni}, Nichols constructs a basis for $H(n)$, proves that its antipode is injective, and then constructs a quotient bialgebra of $H(2)$ which is not a Hopf algebra. In a similar vein, in \cite{Sc}, Schauenburg introduces $H_\infty(n)$ and constructs a quotient Hopf algebra of $H_\infty(4)$ whose antipode is not injective, thus giving the first example of a non-injective surjective antipode. In view of their universal properties, objects such as $H(n)$ and $H_\infty(n)$ are well-suited to be starting points for the construction of counterexamples (as seen above), so it seems worthwhile to gather more information about their structure.

The paper is organized as follows:

In \seref{1} we set up the notations, introduce some preliminary results needed later on, and state our main theorems. 

In \seref{2} Bergman's diamond lemma (\cite{Be}) is used to find bases for the objects of interest $\tilde H(n)$, $n\ge 2$. These bases are somewhat different from those which have appeared in the literature (\cite{Ni, Sc}), and will prove more convenient for our goals. 

In \seref{3} we prove that the Grothendieck rings of finite-dimensional comodules of the Hopf algebras $\tilde H$ are non-commutative polynomial rings. 

\seref{4} contains the main results of this paper, determining the multiplication table of the Grothendieck (semi)ring of $\tilde H$ for all cases except for the $H_1(F)$'s, and recovering the known results on the latter assuming cosemisimplicity. 

\section{Preliminaries}\selabel{1}

We begin by introducing the main conventions and some of the notation, and recalling some generalities on the Hopf algebras alluded to in the previous section. 

We will be working over a fixed base field $k$, which will henceforth be assumed to be algebraically closed. This assumption will simplify things by ensuring, for example, that all simple coalgebras are actually matrix coalgebras. Here, by matrix coalgebra we mean the dual $M_n(k)^*$ of the usual algebra $M_n(k)$ of $n\times n$ matrices over $k$. $M_n(k)^*$ has a basis $\ds (x_{ij})_{i,j=1}^n$ with the coalgebra structure being defined by
\begin{equation}\eqlabel{matrix coalgebra}
\Delta(x_{ij})=\sum_{k=1}^nx_{ik}\otimes x_{kj},\ \varepsilon(x_{ij})=\delta_{ij}, 
\end{equation} 
where $\Delta,\varepsilon$ stand, as usual, for the comultiplication and counit respectively, and $\delta_{ij}$ is the Kronecker symbol. 
The terminology ``matrix coalgebra'' always refers to $M_n(k)^*$ in this paper. A collection of not necessarily linearly independent elements $x_{ij}$ in a coalgebra (bialgebra, Hopf algebra) satisfying \equref{matrix coalgebra} will be referred to as a {\it multiplicative matrix} (following \cite{Ma}). Note that the linear span of a multiplicative matrix is a coalgebra.

We assume familiarity with Hopf algebra theory as appearing, for example, in \cite{Sw, A, Mo}. We also use the standard notations: $\Delta,\varepsilon, S$ for comultiplication, counit and antipode respectively. The words `comodule' and `corepresentation' are used interchangeably, and unless specified otherwise, all comodules are right and finite-dimensional. 

For a Hopf algebra $H$, $\mc M^H$ denotes the category of (finite-dimensional, right) $H$-comodules. The Grothendieck ring of such comodules will be denoted by $K(H)$. Sometimes, when there is no danger of confusion, we might denote a comodule and its representative in the Grothendieck ring by the same symbol. As the category of comodules is left rigid, we have an anti-endomorphism $*$ on $K(H)$, sending the representative of a comodule to the representative of its (left) dual. We might denote the map either by $u\mapsto u^*$ or by $u\mapsto *(u)$. The trivial $H$-comodule will be denoted by $1$; it is the multiplicative identity of the ring $K(H)$.  

In fact, we will also be concerned with the Grothendieck semiring $K_+(H)$, by which we mean the sub-semiring of $K(H)$ generated by the representatives of the comodules. $K_+(H)$ is, of course, invariant under $*$. It is well known that $K(H)$ has a basis (as an abelian group) formed by the set $\mc S=\mc S(H)$ of (isomorphism classes of) simple comodules. There is a natural order on $K$, for which $K_+$ is the positive cone. With this order, $K(H)$ is also a lattice; $\vee$ will denote the supremum operation on this lattice. 

Note that there is a bijection between $\mc S(H)$ and the set of matrix subcoalgebras of $H$, the simple comodule $M$ corresponding to the smallest subcoalgebra $C$ such that the comodule structure map of $M$ factors as 
\[
\rho:M\to M\otimes C\to M\otimes H
\]
(the last map being induced by the inclusion $C\to H$). $C$ is precisely the linear span of the $x_{ij}$, which are uniquely determined by  
\[
\rho(e_j)=\sum_{i=1}^n e_i\otimes x_{ij},\ j=\ol{1,n}.
\]
More generally, the same construction for an $n$-dimensional (not necessarily simple) comodule $M$ yields an $n\times n$ multiplicative matrix in $H$ as soon as we fix a basis $(e_i)_{i=1}^n$ for $M$. In this context, we write $C$ as $C(M)$ and refer to $C$ as the coalgebra corresponding to the comodule $M$.    

The Hopf algebras of interest have already been introduced in the preceding section: they are $H(n)$, the free Hopf algebra on an $n\times n$ matrix coalgebra, $H_\infty(n)$, the free Hopf algebra with bijective antipode on an $n\times n$ matrix coalgebra, and $H_d(F)$, where $d$ is a positive integer and $F\in GL(n,k)$ is an invertible $n\times n$ matrix. $n\ge 2$ will always be assumed, and as stated in the introduction, we use $\tilde H$ (or $\tilde H(n)$ if we want to be more precise) as a generic symbol for any of these Hopf algebras.  

Recall (\cite{Ta,Ni}) that $\tilde H=H(n)$ is defined as follows: one has a multiplicative matrix $\ds X^r=(x^r_{ij})_{i,j}$ for each non-negative integer $r$, satisfying the relations
\begin{equation}\eqlabel{deffree}
\sum_{k=1}^n x^r_{ik}x^{r+1}_{jk}=\delta_{ij}=\sum_{k=1}^n x^{r+1}_{ki}x^r_{kj},\ \forall i,j,r.
\end{equation}
In other words, the transpose $\ds \left(X^{r+1}\right)^t$ is the inverse (in $M_n\left(\tilde H\right)$) of $X^r$. The antipode sends $X^r$ to this transpose, i.e. acts by $S(x^r_{ij})=x^{r+1}_{ji}$. An entirely analogous presentation can be given for $H_{\infty}(n)$, except that this time, $r$ runs through the integers instead of the non-negative integers (see \cite{Sc}).

As for $\tilde H=H_d(F)$, we again have multiplicative matrices $X^r$ as above, but this time $r$ runs through $\mb Z/2d$, the integers modulo $2d$, and the relations \equref{deffree} hold as stated for $r=\ol{0,2d-2}$. For $r=2d-1$ we have instead (in compressed form, using the matrices $X$)
\begin{equation}\eqlabel{deffree 2d}
(X^{2d-1})^{-1}=F(X^0)^tF^{-1}. 
\end{equation}  
That is, instead of making the transpose $(X^0)^t$ the inverse of $X^{2d}$, we ``twist'' by $F$. 

Notice that all the $\tilde H(n)$ have a distinguished $n$-dimensional corepresentation, corresponding to the multiplicative matrix $X^0$: it is a vector space with basis $e_i,\ i=\ol{1,n}$ on which $\tilde H$ acts by
\begin{equation*}\eqlabel{fund}
e_j\mapsto \sum_{i=1}^n e_i\otimes x^0_{ij}. 
\end{equation*}
We refer to this as the {\it fundamental} corepresentation of $\tilde H$, and we will usually denote its representative in $K_+(\tilde H)$ by $f$.  

Finally, whenever we discuss one of the Hopf algebras $\tilde H$, $R=R(\tilde H)$ stands for the set over which the $r$ in the notation $X^r$ used above range: $R=\mb N$, the set of non-negative integers for $\tilde H=H(n)$, $R=\mb Z$ for $\tilde H=H_\infty(n)$, and $R=\mb Z/2d$ when $\tilde H=H_d(F)$.

We can now state the theorems proven in the paper. First, we explain the weaker results, but which hold in greater generality, to be proven in \seref{3}. 

Suppose we are working with $\tilde H$. Consider the free monoid $A_R$ on $R$, with generators $\alpha_r,\ r\in R$, and endow it with the unique anti-endomorphism $*$ sending $\alpha_r$ to $\alpha_{r+1}$ for all $r\in R$. We will refer to the elements of $A_R$ as words in the $\alpha_r$'s, as usual, and for convenience, $\alpha_r$ and $r$ might be identified when there is no danger of confusion. We have a partial order on $A_R$, given by the length of the words. 

There is a unique monoid map $\phi:A_R\to K=K(\tilde H)$ which intertwines the anti- endomorphisms $*$ and sends $\alpha_0$ (for $0\in R$) to the fundamental corepresentation $f$. Now write
\begin{equation}\eqlabel{K}
\phi(x)=\sum_{s\in\mc S'}n_ss+\sum_{s\in\mc S''}n_ss,
\end{equation}    
where $n_s$ are positive integers, and $\mc S''$ is the set of those $s$ which appear in a similar expansion for $\phi(y)$, $y<x$ (i.e. $y\in A_R$ is shorter than $x$). Denote the first sum in the right hand side of \equref{K} by $u_x$. Our first theorem is then the following:

\begin{theorem}\thlabel{freeness}

With $\tilde H$ as above, the map $x\mapsto u_x$ induces a bijection between $A_R$ and $\mc S(\tilde H)$. 

\end{theorem}

In other words, the simple comodules of $\tilde H$ can be labeled in a very natural manner by the elements of the free monoid $A_R$. We will also see in \seref{3} that this easily implies the following:

\begin{corollary}\colabel{freeness}

The Grothendieck ring $K(\tilde H)$ is isomorphic to the free unital algebra $\mb Z[A_R]$ on $R$. 

\end{corollary}

\begin{remark}\relabel{after freeness corollary}

The corollary implies that $K(H(n))$ is isomorphic to $K(H_\infty(m))$, of course ($m,n\ge 2$), since in these two cases we have $R=\mb N$ and $R=\mb Z$. However, the isomorphism appearing in the proof of the corollary will make specific use of these sets $R$, and not just of their cardinality. 

\end{remark}

\seref{4} is concerned with a stronger version of \thref{freeness}, but which does not hold for all $\tilde H$. In order to state it, we need to introduce more notations.  

Let $x\in A_R$. We keep the notation introduced before the statement of \thref{freeness}. Write 
\[
x=r_1r_2\ldots r_n,
\]
where each $r_i$ is one of the letters $\alpha_r$, $r\in R$. Denote by $I(x)$ the set of those $i\in\ol{1,n-1}$ for which $r_ir_{i+1}$ is either of the form $\alpha_r\alpha_{r+1}$ or $\alpha_{r+1}\alpha_r$. For each $i\in I(x)$, denote
\[
x_i=r_1r_2\ldots r_{i-1}r_{i+2}\ldots r_n.
\] 

$\phi$ sends $\alpha_r\alpha_{r+1}$ and $\alpha_{r+1}\alpha_r$ to modules of the form $uu^*$ and respectively $u^*u$ for $u\in K(\tilde H)$, and both of these are $\ge 1$ in $K(\tilde H)$. In conclusion, we get $1\le \phi(r_ir_{i+1})$, and hence $\phi(x_i)\le\phi(x)$ for every $i\in I(x)$. Denote 
\[
u'_x=\phi(x)-\bigvee_{i\in I(x)}\phi(x_i). 
\] 
It's clear that $u'_x\ge u_x$. Our result is the following:

\begin{theorem}\thlabel{maximal freeness}

(a) Suppose $\tilde H$ is not of the form $H_1(F)$. Then, with the notations used above, we have $u'_x=u_x$ for every $x\in A_R$, and hence $x\mapsto u'_x$ is a bijection between $A_R$ and $\mc S(\tilde H)$.

(b) For $\tilde H=H_1(F)$, the statement in (a) is true if and only if $\tilde H$ is cosemisimple.  

\end{theorem}

We now take a moment to recall the situation in the literature for the free cosovereign Hopf algebras $H_1(F)$, and make the connection between those results and the theorems stated above. 

In \cite{Ba} the free monoid $A$ on two generators $\alpha,\beta$ is considered, with the involution $*$ used above in the more general situation; here, this involution simply interchanges $\alpha$ and $\beta$.  B\u anic\u a then introduces a new product $\odot$ on the monoid ring $\mb Z[A]$:
\begin{equation}\eqlabel{Banica product}
x\odot y=\sum_{x=ag,y=g^*b}ab,\ x,y\in A.
\end{equation}
It is shown that this is indeed an associative product, and moreover, $(\mb Z[A],\odot)$ is again the free ring generated by $\alpha,\beta$. 

The results in \cite{Bi4} which are relevant here can be rephrased and summarized as follows (\cite[Theorem 1.1,(iii)]{Bi4}):

\begin{theorem}\thlabel{Bichon}

Assume $k$ has characteristic zero and $\tilde H=H_1(F)$ is cosemisimple. Then, the map $(\mb Z[A],\odot)\to K(\tilde H)$ defined by sending $\alpha$ and $\beta$ to $f$ and $f^*$ respectively is an isomorphism of rings with involution, and induces a bijection of $A$ with the set of isomorphism classes of irreducible corepresentations. 

\end{theorem}

Note that this generalizes \cite[Th\'eor\`eme 1 (i)]{Ba}, and so includes the corepresentation theory of Wang and Van Daele's universal compact quantum groups mentioned in the introduction. Bichon actually determines exactly when a universal cosovereign Hopf algebra is cosemisimple in characteristic zero, but we do not make use of that result here. 

It is not difficult to see that part (b) of \thref{maximal freeness} (in characteristic zero) is, in fact, another way of stating \thref{Bichon}. 

For $\tilde H=H(n)$, \thref{maximal freeness} says, essentially, that the Grothendieck ring $K(\tilde H)$ is generated as a ring with anti-endomorphism by the fundamental corepresentation $f$, and the relations satisfied by the generators $f,f^*,f^{**}$, etc. are precisely those imposed by the fact that $\mc M^H$ is a left rigid monoidal category, and nothing more. In other words, $K(\tilde H)$ is ``as free as possible'' on the dual iterates $f,f^*,f^{**}$, etc. of $f$. We refer to this situation as ``maximal freeness'', hence the title of \seref{4}. 

The meaning of \thref{maximal freeness} for $\tilde H=H_\infty(n)$ or $\tilde H=H_d(F)$ is similar: in the first case $K(\tilde H)$ is maximally free on the iterates $*^r(f)$, $r\in R=\mb Z$ under the constraints that $\mc M^H$ be a rigid (both left and right) monoidal category, while for $\tilde H=H_d(F)$, in the good cases (i.e. when either $d>1$ or $d=1$ and $H_1(F)$ is cosemisimple), $K$ is maximally free on the dual iterates of $f$ under the constraint that $\mc M^H$ be a rigid monoidal category for which the $2d$'th power of the dual is naturally isomorphic to the identity functor.    

\section{Putting the diamond lemma to good use}\selabel{2}

As announced in the introduction, in this section we will look at the Hopf algebras $\tilde H$ in more detail, and bases over $k$ will be constructed for them using Bergman's diamond lemma. We use the results and language in \cite{Be} freely, and refer to that paper for the necessary background and terminology.

Typically, we won't go through the actual verification of the fact that the ambiguities we get (\cite{Be}) are resolvable. Instead, for the more formidable ambiguities, we give an argument which simplifies the situation considerably and makes the verification itelf more or less trivial.  

A basis for $H(n)$ was constructed by Nichols in \cite{Ni}, and the technique was adapted to $H_\infty(n)$ in \cite{Sc}. We stated in \cite{Ch} that an analogous approach works for what here would be called $H_d(I_n)$. Because the result will be different here, we recall only that the bases used in these papers consisted of all words in the generators $x^r_{ij}$ (introduced in the previous section) which contain no subwords of either one of the forms
\[
x^r_{in}x^{r+1}_{jn},\quad x^{r+1}_{ni}x^r_{nj},\quad x^r_{in}x^{r+1}_{jn-1}x^{r+2}_{kn-1} ,\quad x^{r+2}_{ni}x^{r+1}_{n-1j}x^r_{n-1k}, 
\]
for $r$ ranging through $R=R(\tilde H)$. 

Let us now look at $\tilde H=H(n)$, $H_\infty(n)$, or $H_d(F)$, with $F\in GL(n,k)$. The following notation will be useful: bold symbols such as ${\bf r}=(r_1,\ldots,r_k)$ and ${\bf i}=(i_1,\ldots,i_k)$ denote vectors of elements $r_j\in R$ and $i_j\in \ol{1,n}$ respectively. The length of the vector ${\bf r}$ will be denoted by $|{\bf r}|$. $x^{\bf r}_{\bf ij}$ denotes the product $x^{r_1}_{i_1j_1}\ldots x^{r_k}_{i_kj_k}$; $x^{\bf r}_{\bf ij}$ will also occasionally be referred to as a monomial of type ${\bf r}$.  

In order to apply the diamond lemma, we need a collection of reductions, and a semigroup partial order on the monoid $\la \mc X\ra$ freely generated by the set $\mc X$ of symbols $x^r_{ij}$, $r\in R$ and $i,j\in\ol{1,n}$. We take care of the ordering later; the reductions are as follows: 

\begin{equation}\eqlabel{red1}
x^r_{in} x^{r+1}_{jn} \to \delta_{ij}-\sum_{a<n}x^r_{ia} x^{r+1}_{ja},\quad \mbox{r even}
\end{equation}

\begin{equation}\eqlabel{red2}
x^r_{i1} x^{r+1}_{j1} \to \delta_{ij}-\sum_{a>1}x^r_{ia} x^{r+1}_{ja},\quad \mbox{r odd}
\end{equation}

\begin{equation}\eqlabel{red3}
x^{r+1}_{ni} x^r_{nj} \to \delta_{ij}-\sum_{a<n}x^{r+1}_{ai} x^r_{aj},\quad \mbox{r odd}
\end{equation} 

\begin{equation}\eqlabel{red4}
x^{r+1}_{1i} x^r_{1j} \to \delta_{ij}-\sum_{a>1}x^{r+1}_{ai} x^r_{aj},\quad \mbox{r even}
\end{equation}
Here $\delta_{ij}$ is the Kronecker delta, and since $R$ is one of the sets $\mb N$, $\mb Z$ or $\mb Z/2d$, it makes sense to talk about even and odd elements $r\in R$.    

These reductions, with $r$ ranging through the whole set $R$, account for all the relations defining the algebras $H(n)$ and $H_\infty(n)$ (and even $H_d(I_n)$). So by the diamond lemma, in order to conclude that the monomials which contain no subwords as in the left hand sides of \equref{red1} - \equref{red4} form a basis in these cases, it suffices to prove (once the semigroup partial order with the descending chain condition and compatible with the reductions has been found) that all resulting overlap and inclusion ambiguities are resolvable. 

The advantage of this choice of reductions over those in \cite{Ni, Sc, Ch} is the fact that now there is essentially only one ambiguity to resolve (``essentially'' meaning up to interchanging $1$ and $n$, a translation of $R$, etc.). This essentially unique (overlap) ambiguity is $x^r_{in} x^{r+1}_{1n} x^r_{1j}$ for even $r$, and one sees easily that it is indeed resolvable. Hence, we now have a basis for $H(n)$ and $H_\infty(n)$.

In order to treat $H=H_d(F)$, the arbitrary invertible matrix $F$ must be brought into the picture. Recall (\equref{deffree}) that as an algebra, $H$ is generated by the elements $x^r_{ij}$ for $r\in\mb Z/2d=\ol{0,2d-1}$, and $i,j\in\ol{1,n}$, subject to the relations
\[
(X^{r+1})^t=(X^r)^{-1},\ \forall r\in\ol{0,2d-2},
\]
\[
F(X^0)^tF^{-1}=(X^{2d-1})^{-1}.
\]
Here, $X^r$ is the matrix $(x^r_{ij})_{i,j}\in M_n(H)$, and the superscript $^t$ denotes the transpose of an $n\times n$ matrix.   
    
To get reductions which account for all of this, we first make the observation that it suffices to consider the case when $F$ is upper triangular. More precisely, we have an isomorphism $H_d(F)\cong H_d(PFP^{-1})$ for any $P\in GL(n,k)$, and any matrix can be made upper triangular by conjugation (the field is algebraically closed!). 

The claim about the isomorphism is proven in \cite{Bi2} for $d=1$, i.e. for the free cosovereign Hopf algebras. It suffices to send $X^0$ from $H_d(PFP^{-1})$ to $(P^t)^{-1}X^0P^t$ from $H_d(F)$, and this is easily seen to extend to a Hopf algebra isomorphism for the Hopf algebra structures described in the previous section. Hence, from now on, whenever $H_d(F)$ comes up, we assume that $F$ is upper triangular. With this assumption in place, we keep the reductions \equref{red1} - \equref{red4} for $r=\ol{0,2d-2}$, and add the two reductions

\begin{equation}\eqlabel{red5}
x^{2d-1}_{i1} x^0_{j1} \to F_{11}^{-1}F_{jj}\left(\delta_{ij}-\sum_{(l,p,u)\ne (1,1,j)} F_{lp}(F^{-1})_{uj}x^{2d-1}_{il} x^0_{up}\right)
\end{equation}

\begin{equation}\eqlabel{red6}
x^0_{ni} x^{2d-1}_{nj} \to F_{ii}^{-1}F_{nn}\left(\delta_{ij}-\sum_{(p,u,l)\ne (i,n,n)}x^0_{up} x^{2d-1}_{lj}\right)
\end{equation}

We have postponed tackling the issue of the semigroup partial order on $\la \mc X\ra$ until now because we would like to find such an order which is compatible with all of our reductions \equref{red1} - \equref{red6} at once (in addition to having the descending chain condition). For our purposes, the following works.

First, words in the $x^r_{ij}$ are ordered according to their length (that is, shorter words are smaller). Then, among words of the same length, we only compare pairs of the form $x^{\bf r}_{\bf ij}$, $x^{\bf r}_{\bf i'j'}$ (i.e. with the same vector ${\bf r}$). So consider such a pair, say
\[
x^{\bf r}_{\bf ij}=x^{r_1}_{i_1j_1}\ldots x^{r_k}_{i_kj_k},\quad x^{\bf r}_{\bf i'j'}=x^{r_1}_{i'_1j'_1}\ldots x^{r_k}_{i'_kj'_k}.
\]  

Let $\ell$ be the smallest index for which the pairs $(i_\ell,j_\ell)$ and $(i'_\ell,j'_\ell)$ are different. Then, the order between our monomials $x^{\bf r}_{\bf ij}$ and $x^{\bf r}_{\bf i'j'}$ is the same as the order between the two-term monomials $x^{\bf s}_{\bf uv}$ and $x^{\bf s}_{\bf u'v'}$ respectively, where 
\[
{\bf s}=(r_\ell,r_{\ell+1}),
\]
\begin{align*}
\qquad \qquad \qquad \qquad \qquad \qquad 
{\bf u}  &=(i_\ell,i_{\ell+1}),    &  {\bf v}    &=(j_\ell,j_{\ell+1}), 
\qquad \qquad \qquad \qquad \qquad \qquad   \\
\qquad \qquad \qquad \qquad \qquad \qquad 
{\bf u'} &=(i'_\ell,i'_{\ell+1}),  &  {\bf v'}   &=(j'_\ell,j'_{\ell+1}). 
\qquad \qquad \qquad \qquad \qquad \qquad 
\end{align*}
The order is undefined if $\ell=k$, i.e. the monomials are incomparable in our partial order in this case. 

The above is clearly a semigroup partial order for any partial order whatsoever on the two-term monomials, so it suffices to describe that. We simply make the two-term monomials on the left hand side of each of \equref{red1} - \equref{red6} greater than any two-term monomial in the right hand side of the same reduction; it is not difficult to see that this can be extended to a partial order on the two-term monomials. 

For example, if ${\bf r}=(r,r\pm 1)$ and $r$ is even, then the order can be defined as follows:  
\[
x^{\bf r}_{\bf ij}>x^{\bf r}_{\bf i'j'}\quad \mbox{if}\quad {\bf i}=(n,n)\ne {\bf i'},
\]
\[
x^{\bf r}_{\bf ij}>x^{\bf r}_{\bf i'j'}\quad \mbox{if}\quad {\bf i}=(n,n)={\bf i'},\quad {\bf j'}\ne (n,n),\quad {\bf ij}<{\bf i'j'}\ \mbox{lexicographically},
\]
\[
x^r_{in}x^{r\pm 1}_{jn}>x^r_{ia}x^{r\pm 1}_{ja},\ \forall a<n,\ i,\ j. 
\]
Here, ${\bf ij}$ is simply the concatenation of the vectors ${\bf i}$ and ${\bf j}$. In checking that this works, one must make use of the fact that our matrix $F$ is now assumed to be upper triangular. A similar arrangement works for ${\bf r}=(r,r\pm 1)$ with odd $r$, and this is enough for our purposes. 

Apart from the ambiguities resulting from the reductions \equref{red1} - \equref{red4} (for $r=\ol{0,2d-2}$), which are easily checked to be resolvable, we must also consider the ambiguities of the form $x^0_{nj} x^{2d-1}_{n1} x^0_{i1}$ and $x^{2d-1}_{i1} x^0_{n1} x^{2d-1}_{nj}$. Because of the complicated form of the reductions \equref{red5}, \equref{red6}, it is much more cumbersome to check the resolvability of these. We will make use of a trick to reduce \equref{red5} and \equref{red6} to the case when $F$ is diagonal; this simplifies the task of checking the resolvability significantly, and we leave that task to the reader. 

The trick alluded to in the previous paragraph is of the following nature: (1) first, we would like to conclude that the desired resolvability depends only on the conjugacy class of $F$ in the group $T(n,k)$ of upper triangular $n\times n$ matrices; (2) next, we observe that it suffices to prove the resolvability only for $F$ in a Zariski dense subset of $T(n,k)$. These two steps would indeed reduce the checking to the case when $F$ is diagonal, because we can take our Zariski dense set to be that of diagonalizable upper triangular matrices.

To prove step (1), notice that by the diamond lemma, the resolvability can be regarded as a statement about the dimension of the span of the $x^{\bf r}_{\bf ij}$ in $H_d(F)$, where ${\bf r}$ is either $(0,2d-1,0)$ or $(2d-1,0,2d-1)$. But by the argument used to prove the isomorphism $H_d(F)\cong H_d(PFP^{-1})$, this dimension depends only on the conjugacy class of $F$ in $T(n,k)$.

For step (2), let us focus on resolving $x^0_{nj} x^{2d-1}_{n1} x^0_{i1}$ (the other ambiguity being essentially the same). We can either apply \equref{red6} to the first two factors and then \equref{red5} to every term in the resulting sum for which it applies, or apply \equref{red5} to the last two factors and then \equref{red6} to all the terms to which it applies in the resulting sum. The aim is to prove that if for a Zariski dense subset of $T(n,k)$ the resulting expressions are identical, then they are identical for all $F$. But this is clear: the resulting expressions are linear combinations of terms of the form $x^{\bf r}_{\bf ij}$ for ${\bf r}=(0,2d-1,0)$, and the coefficients of each such term are regular functions defined on the algebraic variety $T(n,k)$; if these coefficients coincide on a Zariski dense subset of $T(n,k)$, they coincide everywhere by continuity.

We now summarize the conslusions of this section:

\begin{proposition}\prlabel{diamond}

(a) For $\tilde H=H(n)$ or $H_\infty(n)$, the diamond lemma is applicable to the reductions \equref{red1} - \equref{red4} (for $r\in R(\tilde H)$), so the words in $x^r_{ij}$ containing no subwords as in the left hand sides of those reductions form a basis for $\tilde H$.  

(b) Let $F\in T(n,k)$. For $\tilde H=H_d(F)$, the same conclusion as in (a) holds, with the reductions \equref{red1} - \equref{red4}, $r=\ol{0,2d-2}$ and \equref{red5}, \equref{red6}. 

\end{proposition}

The expansion of an element of $\tilde H$ as a linear combination of the basis given here will be referred to as the {\it standard form} of the element. Similarly, the standard form of an element of $\tilde H\otimes\tilde H$ is its expansion as a linear combination of tensor products of reduced monomials. The terms reducible/irreducible for monomials $x^{\bf r}_{\bf ij}$ as above always refer to the reductions \equref{red1} - \equref{red6}. 

Finally, note that $\tilde H$ is filtered by the non-negative integers, with $\tilde H_k$ being the span of the monomials $x^{\bf r}_{\bf ij}$ for $|{\bf r}|\le k$. 

\section{Freeness}\selabel{3}

In this section we prove \thref{freeness} and its consequence, \coref{freeness}. Let us take care of the corollary first, assuming the theorem is proven. 

We introduce some more notation first: given $r\in R=R(\tilde H)$, $f_r\in K=K(\tilde H)$ denotes the comodule of $\tilde H$ corresponding to the matrix coalgebra $X^r$. Similarly, given a vector ${\bf r}=(r_1,\ldots,r_k)$ with entries in $R$, $f_{\bf r}$ denotes the product $f_{r_1}\ldots f_{r_k}$. Similarly, $X^{\bf r}$ denotes the product of the coalgebras $X^{r_i}$; it is the coalgebra $C(f_{r_i})$ corresponding to the tensor product of the comodules $f_{r_i}$ (in the same order $r_1,r_2,\ldots$).  

Since the words $x\in A_R$ are clearly in one-to-one correspondence with the vectors ${\bf r}$ with entries in $R$, we may denote the elements $u_x,u'_x$ introduced in \seref{1} by $u_{\bf r}$ and $u'_{\bf r}$ respectively (for the vector ${\bf r}$ corresponding to $x$).

\renewcommand{\proofname}{Proof of \coref{freeness}}
\begin{proof}

Recall the morphism $\phi:\mb Z[A_R]\to K=K(\tilde H)$ of rings endowed with an anti-endomorphism introduced in \seref{1}. Both the free unital ring $\mb Z[A_R]$ on $R$ and the Grothendieck ring $K$ are filtered: the former by the length of the words on $R$, and the latter by setting, $K_n$ equal to the linear combination of those simple comodules which are $\le f_{\bf r}$ for some vector $\bf r\subset R$ of length $\le n$ for each non-negative integer $n$ (remember that there is an order on $K$, with $K_+$ as a positive cone).   

The map $\phi$ from \seref{1} preserves the filtration, and \thref{freeness} says precisely that the induced graded map between associated graded rings is an isomorphism. But this implies that $\phi$ itself is bijective, and we are done. \end{proof}
\renewcommand{\proofname}{Proof}

\begin{remark}

The corollary generalizes \cite[Corollary 5.5]{Bi4}, which consists of the corresponding statement for the cosemisimple universal cosovereign Hopf algebras $H_1(F)$ in characteristic zero. 

\end{remark}

Before going into the proof of the theorem, we make several preliminary observations on the problem. One of these is the following reformulation:

\begin{lemma}\lelabel{freeness}

\thref{freeness} is equivalent to the fact that the elements $u_{\bf r}\in K(\tilde H)$ appearing in its statement are simple. 

\end{lemma}

\begin{proof}

That the $u_{\bf r}$ are simple is part of the statement of \thref{freeness}, so we only need the opposite implication. Hence, we now assume that all $u_{\bf r}$ are simple. 
   
Since the Hopf algebra $\tilde H$ is the sum of the subcoalgebras $X^{\bf r}$ (for vectors ${\bf r}$ with entries in $R$), it follows that its comodules are subcomodules of the tensor products (represented by) the $f_{\bf r}$. Now consider (the representative of) a simple comodule $u\in K=K(\tilde H)$. We have just noticed that we must have $u\le f_{\bf r}$ in $K$ for some vector ${\bf r}$; choose such an ${\bf r}$ of the smallest length possible. It then follows from the definition of the $u_{\bf s}$'s that $u=u_{\bf r}$; consequently, $\phi$ is a surjection of $A_R$ on $\mc S(\tilde H)$. 

On the other hand, again from the definition of $u_{\bf r}$, it follows that the elements of the corresponding matrix subcoalgebra of $\tilde H$, in their standard form, contain reduced monomials of type ${\bf r}$ (apart from those of type ${\bf s}$ for $|\bf s|<|{\bf r}|$). But this immediately implies that the $u_{\bf r}$ are all different, so $\phi$ is also injective.    
\end{proof}

The previous lemma allows us to focus on proving that $u_{\bf r}$ are all simple. In order to state the next preliminary result, we introduce the following terminology: a vector ${\bf r}=(r_1,\ldots,r_k)\subset R$ is said to be a {\it 1-step} vector if $r_{i+1}=r_i\pm 1$ for all $i$. The claim is now the following:

\begin{lemma}\lelabel{1-step}

If $u_{\bf r}$ is simple for every 1-step vector ${\bf r}\subset R$, then all $u_{\bf r}$ are simple. 

\end{lemma}

\begin{proof}

We prove (under the hypothesis of the lemma) that all $u_{\bf r}$ are simple by induction on the length of ${\bf r}$. Vectors of length $1$ (or $0$, i.e. the empty vector) are by definition 1-step, so the base case of the induction is taken care of. Now fix a vector ${\bf r}$, and assume the statement is proven for all shorter vectors. 

If ${\bf r}$ is 1-step, there is nothing to prove. Otherwise, we can write ${\bf r}$ as a concatenation ${\bf r}_1{\bf r}_2$, where ${\bf r}_1$ and ${\bf r}_2$ are vectors such that the last entry $r_1$ of ${\bf r}_1$ and the first entry $r_2$ of ${\bf r}_2$ satisfy $r_2\ne r_1\pm 1$. 

By the induction hypothesis, the coalgebras $C_i$, $i=1,2$ corresponding respectively to $u_{{\bf r}_i}$ are matrix coalgebras; since the intersection of $C_i$ with the matrix coalgebra $X^{\bf s}$ for ${\bf s}$ shorter than ${\bf r}_i$ is trivial, the projection of $C_i$ on the span of the monomials of type ${\bf r}_i$ (respectively) obtained by sending all other monomials to zero is injective. But the form of the basis in \prref{diamond} makes it clear that the product of two irreducible monomials of types ${\bf r}_1$ and respectively ${\bf r}_2$ is again irreducible. This, together with the previous observation, implies that the multiplication map from the tensor product $C_1\otimes C_2$ to the product $C=C_1C_2$ inside $\tilde H$ is an isomorphism, and hence that (a) $u_{\bf r}=u_{{\bf r}_1}u_{{\bf r}_2}$, and (b) $u_{\bf r}$ is simple, with matrix coalgebra $C$. This completes the induction step. 
\end{proof}

In the proof of \thref{freeness}, we will deal separately with the universal cosovereign Hopf algebras $H_1(F)$. For the other cases, $\tilde H=H(n)$, $H_\infty(n)$ or $H_d(F)$ for some $d>1$, the following observation will be useful:

\begin{lemma}\lelabel{not cosovereign}

If \thref{freeness} holds for $\tilde H=H(n)$, then it holds for $\tilde H=H_\infty(n)$ or $\tilde H=H_d(F)$, $d>1$.  

\end{lemma}

\begin{proof}

By the two previous lemmas, it is enough to check that $u_{\bf r}$ is simple for any 1-step vector ${\bf r}$. 

Assume first that $\tilde H=H_\infty(n)$. In this case, by applying a high enough power of the antipode (which is bijective), we may as well assume that integer entries of ${\bf r}$ are, in fact, non-negative. But the bases for our Hopf algebras given by \prref{diamond} make it clear that the map $H(n)\to H_\infty(n)$ sending $x^0_{ij}$ in $H(n)$ to $x^0_{ij}$ in $H_\infty(n)$ induces an isomorphism of $K(H(n))$ onto the subring of $K(H_\infty(n))$ generated by the subcomodules of the $f_{\bf r}$'s for non-negative vectors ${\bf r}$.  

Now take $\tilde H=H_d(F)$ for some $d>1$ and $F\in GL(n,k)$. We have a surjective Hopf algebra map $H(n)\to H_d(F)$, sending $x^r_{ij}$ in $H(n)$ to $x^{\bar r}_{ij}$ in $H_d(F)$, where $r\mapsto \bar r$ is the obvious surjection $\mb N\to \mb Z/2d$. If we prove that the matrix coalgebra $C_{\bf r}$ corresponding to $u_{\bf r}\in K(H(n))$ gets mapped to a matrix coalgebra, then we are done. 

It is clear from the reductions \equref{red1} - \equref{red6} that whenever ${\bf r}\subset\mb N$ is a 1-step vector, a reduced monomial of type ${\bf r}$ in $H(n)$ is mapped onto a reduced word of type ${\bf \bar r}\subset\mb Z/2d$ in $H_d(F)$ as long as $d>1$. In other words, the span of the reduced words of type ${\bf r}$ is mapped injectively into $H_d(F)$. In view of the fact (also noted in the previous proof) that the projection onto the span of the words of type ${\bf r}$ obtained by sending all other monomials to zero is injective on the matrix coalgebra $C_{\bf r}$, this concludes the proof. 
\end{proof}

For $H_1(F)$ we will have to make use of Bichon's results on Hopf-Galois systems (\cite{Bi3}, \cite[Proposition 2.1, 2.4]{Bi4}): what is relevant for us here is that if $F$ is upper triangular with diagonal $D$, then there is an equivalence of monoidal categories between $H_1(F)$ and $H_1(D)$ matching up the fundamental corepresentations. Hence, when dealing with $H_1(F)$ in the proof, we can (and will) assume that $F$ is diagonal. With this assumption in place, the proof below will take care of all the possibilities for $\tilde H$ at once.

\renewcommand{\proofname}{Proof of \thref{freeness}}
\begin{proof}

The following argument applies to $\tilde H=H(n)$ or $H_1(F)$ for some diagonal invertible matrix $F\in GL(n,k)$ (see the comments above). Recall that $n\ge 2$. \leref{not cosovereign} says that we will then get the cases $\tilde H=H_\infty(n)$ or $H_d(F)$, $d>1$ for free, so this suffices to prove the theorem. Furthermore, by \leref{freeness}, we only have to prove that the comodules $u_{\bf r}$ are simple.  

Fix an $R$-vector ${\bf r}=(r_1,\ldots,r_k)$. Let $C$ be a simple (hence matrix) subcoalgebra of $C_{\bf r}=C(u_{\bf r})$. Denote by ${\bf \ell}$ the alternating vector $(1,n,1,n,\ldots)$, of length $|{\bf r}|$ (we could have used any two different elements of $\ol{1,n}$ instead of $1$ and $n$). I claim that $C$ necessarily contains an element $x$ whose standard form contains the monomial $x^{\bf r}_{\bf \ell\ell}$. 

Assuming the claim for now, the proof continues as follows. Consider the Hopf algebra $H$, obtained as a quotient of $H(n)$ by sending all off-diagonal generators $x^0_{ij}$, $i\ne j$ to zero. $H$ is nothing but the group algebra of the free group $F_n$ on the $n$ generators $x_i=x^0_{ii}$, $i=\ol{1,n}$. Because in this proof $\tilde H$ is $H(n)$ or $H_1(F)$ for a {\it diagonal} matrix $F$, the surjection $H(n)\to H$ factors through $\tilde H$. Hence, we now have a surjection $\psi:\tilde H\to H$, obtained by sending all off-diagonal generators $x^0_{ij}$, $i\ne j$ to zero. The induced map on Grothendieck rings will also be denoted by $\psi$. 

Because $x^{\bf r}_{\bf \ell\ell}$ has non-zero coefficient in $x\in C$, it follows that the simple $\tilde H$-comodule corresponding to $C$, when regarded as an $H$-comodule by ``scalar corestriction'' via $\psi$, contains the $1$-dimensonal $H$-comodule $v$ corresponding to $x_1^{\varepsilon_1}x_n^{\varepsilon_2}x_1^{\varepsilon_3}\ldots$ as a summand, where the expression contains $|{\bf r}|$ factors, and $\varepsilon_i=1$ if $r_i$ is even and $-1$ otherwise. $C$ was an arbitrary matrix subcoalgebra; unless $u_{\bf r}$ is simple, this means that $2v\le \psi(u_{\bf r})$ (in the usual order on the Grothendieck ring $K(H)$). This, however, is plainly false: on the one hand we have $u_{\bf r}\le f_{\bf r}$ in $K(\tilde H)$ (recall that $f_{\bf r}=f_{r_1}\ldots f_{r_2}$), and on the other hand, $\psi(x^{\bf r}_{\bf ij})$ is equal to $x_1^{\varepsilon_1}x_n^{\varepsilon_2}x_1^{\varepsilon_3}\ldots$ for precisely {\it one} (reducible or irreducible) monomial $x^{\bf r}_{\bf ij}$ of type ${\bf r}$, which means that $2v\not\le \psi(f_{\bf r})$ in $K(H)$. 

It remains to prove the claim that $x^{\bf r}_{\bf \ell\ell}$ has non-zero coefficient in the standard form of some element of $C$. The following technique was used in the proof of \cite[Proposition 2.6]{Ch}, as well as several other results in that paper. 

Consider any non-zero element $x$ of $C$. Because $C\subset X^{\bf r}$ and the intersection of $C$ with any coalgebra of the form $X^{\bf s}$, $|{\bf s}|<|{\bf r}|$ is trivial, the standard form of $x$ must contain some reduced monomial $x^{\bf r}_{\bf ij}$. Using the comultiplication 
\[
\Delta(x^r_{ij})=\sum_{a=1}^n x^r_{ia}\otimes x^r_{aj},
\]  
we conclude that the standard form of $\Delta(x)$ contains $x^{\bf r}_{\bf i\ell}\otimes x^{\bf r}_{\bf \ell j}$ (one sees easily that both $x^{\bf r}_{\bf i\ell}$ and $x^{\bf r}_{\bf \ell j}$ must be reduced if $x^{\bf r}_{\bf ij}$ is). But that the standard form of some element of $C$ (which we may as well assume is our $x$) contains $x^{\bf r}_{\bf i\ell}$. Now simply repeat the argument to conclude that $x^{\bf r}_{\bf \ell\ell}$ is indeed contained in the standard form of some element of $C$.   
\end{proof}
\renewcommand{\proofname}{Proof}

\section{Maximal freeness}\selabel{4}

The goal in this section is to prove \thref{maximal freeness}. We begin by noticing that the lemmas in the previous section have analogues which apply here almost word for word. 

The first observation is that since we now know that $u_{\bf r}$ are simple and it we remarked in \seref{1} that $u_{\bf r}\le u'_{\bf r}$ in $K(\tilde H)$, the result that $u'_{\bf r}=u_{\bf r}$, which is what we're after in \thref{maximal freeness}, is equivalent to saying that $u'_{\bf r}$ being simple. This is an analogue of \leref{freeness}. In each particular case, we use whichever formulation seems more convenient.  

\leref{1-step} can also be adapted to $u'_{\bf r}$:

\begin{lemma}\lelabel{1-step'}

Let $\tilde H$ be one of our Hopf algebras, and $R=R(\tilde H)$, as usual. If $u'_{\bf r}=u_{\bf r}$ for every 1-step $R$-vector ${\bf r}$, then the same holds for all vectors ${\bf r}$. 

\end{lemma}

\begin{proof}

We will adapt the proof of \leref{1-step}, using induction on $|{\bf r}|$ again. If ${\bf r}$ is not 1-step, then write it as a concatenation ${\bf r}_1{\bf r}_2$, as in that proof. By the induction hypothesis we know that $u'_{{\bf r}_i}=u_{{\bf r}_i}$, $i=1,2$, so the argument used in the proof of \leref{1-step} shows that the tensor product $u'_{{\bf r}_1}u'_{{\bf r}_2}$ is simple. Since it's easy to see from the definition of the $u'_{\bf s}$'s that $u'_{\bf r}\le u'_{{\bf r}_1}u'_{{\bf r}_2}$, we get the desired result that $u'_{\bf r}$ is simple. 
\end{proof}

The following analogue of \leref{not cosovereign} will come in handy in the proof of \thref{maximal freeness}, (a). Once more, the proof of \leref{not cosovereign} can be adapted immediately to the present situation.

\begin{lemma}\lelabel{not cosovereign'}

If $u'_{\bf r}=u_{\bf r}$ for $\tilde H=H(n)$ and all $R(\tilde H)$-vectors ${\bf r}$, then the same is true for $\tilde H=H_\infty(n)$ or $H_d(F)$, $d>1$. 

\end{lemma}

\thref{maximal freeness} (a) has now been reduced to the case $\tilde H=H(n)$. We reduce it further to $\tilde H=H(2)$ by the following observation: it was shown in \cite[Corollary 5.3]{Bi3} that there is a monoidal equivalence between the categories of comodules of $H(n)$ and $H(2)$ for every $n\ge 2$. Furthermore, it follows from the discussions in that paper that this equivalence matches up the fundamental corepresentations. Since the statement of \thref{maximal freeness} clearly depends only on the Grothendieck ring (as a ring endowed with an anti-endomorphism) and the choice of a distinguished element of that ring (the fundamental corepresentation), we can indeed work only with $H(2)$.    

We now need to go into the combinatorics of the multiplication in $K(\tilde H)$ in more detail, and this requires yet more new terminology and notations. It will be very useful to know the dimensions of (the comodules represented by) the $u'_{\bf r}$'s, so we begin by introducing the notations necessary to state that result.

Fix our Hopf algebra $\tilde H=H(n)$, $H_\infty(n)$, or $H_d(F)$ for some $F\in GL(n,k)$. Let 
\[
{\bf r}=(r_1,\ldots,r_k)
\] 
be a vector with entries in $R=R(\tilde H)$, as usual. Now consider sequences $n_1,\ldots,n_k$ of positive integers in the range $\ol{1,n}$ with the properties that (a) if $r_i$ is even and $r_{i+1}=r_i\pm 1$, then the pair $(n_i,n_{i+1})$ is different from $(n,n)$, and (b) if $r_i$ is odd and $r_{i+1}=r_i\pm 1$, then $(n_i,n_{i+1})\ne (1,1)$. Denote by ${\mc O}_{\bf r}$ the collection of such vectors, and by $n_{\bf r}$ the cardinality of ${\mc O}_{\bf r}$.

\begin{remark}\relabel{n_r}

A quick look at the reduction formulas \equref{red1} - \equref{red6} shows that when $R=\mb Z/2$ (i.e. $\tilde H$ is one of the universal cosovereign Hopf algebras $H_1(F)$), the number of irreducible monomials of type ${\bf r}$ is precisely $n_{\bf r}^2$. This observation will be crucial in the proof of \thref{maximal freeness}. 

\end{remark}

It will be seen below (\coref{table}) that the dimension of $u'_{\bf r}$ is precisely $n_{\bf r}$, and at the same time, we will see how the basic tensor products $f_{\bf r}=f_{r_1}\ldots f_{r_k}$ decompose as sums of $u'_{\bf s}$'s. The following setup is relevant for the latter purpose. 

For a vector ${\bf r}=(r_i,\ i=\ol{1,k})$ as above, we introduce the following notion:

\begin{definition}\delabel{conf}

An ${\bf r}$-{\it configuration} is a sequence of length $k=|{\bf r}|$ of symbols, with each symbol being either empty (i.e. no symbol at all) or one of the parantheses `$($', `$)$', according to the following rules:

(a) the sequence of symbols is grammatically correct as a sequence of parantheses;

(b) if $|{\bf r}|=0,1$, then the only ${\bf r}$-configuration is the empty one (only the empty symbol, or in other words, no symbols at all); 

(c) if we have a $($ at position $i$ and its pair $)$ at $j>i$, then $r_j=r_i\pm 1$;

(d) if we have a $($ at $i$ and its pair $)$ at $j>i$, then all positions between $i$ and $j$ are filled up completely with paired up parantheses (in particular, it follows that $j-i$ is odd). 

\end{definition}

The collection of all ${\bf r}$-configurations will be denoted by ${\rm Conf}_{\bf r}$, with $\emptyset$ standing for the empty configuration. We give some examples to help clarify the definition. The parantheses appear above their positions, with nothing appearing over the positions corresponding to the empty symbol.   

Suppose ${\bf r}=(1,2,1)$. Apart from the empty configuration, we have two more, namely  
\begin{align*}
&(&    &)&    & &             \qquad\qquad    & &   &(&   &)&       \\
&1&    &2&    &3&   {\rm and} \qquad\qquad    &1&   &2&   &3&. 
\end{align*}
Similarly, if ${\bf r}=(1,2,1,2)$, then there are five non-empty ${\bf r}$-configurations. Those with only one pair of parantheses are
\begin{align*}
& & & & &(& &)&          &(& &)& & & & &           & & &(& &)& & &        \\
&1& &2& &3& &4&, \qquad  &1& &2& &3& &4&, \qquad    &1& &2& &3& &4&,
\end{align*}
while those with two pairs of parantheses are 
\begin{align*}
&(& &)& &(& &)&                            &(& &(& &)& &)&             \\
&1& &2& &3& &4& {\rm and} \qquad\quad      &1& &2& &3& &4&.
\end{align*}
Given a vector ${\bf r}$ and an ${\bf r}$-configuration $c\in {\rm Conf}_{\bf r}$, we denote by ${\bf r}_c$ the vector obtained from ${\bf r}$ by removing the entries whose positions hold parantheses in $c$.

\begin{remark}\relabel{trans}

Note that ${\bf r}$-configurations enjoy is a certain ``transitivity'' property: suppose $($ occupies position $i$ in $c$, while its pair $)$ occupies position $i+1$. Let $d$ be the ${\bf r}$-configuration consisting of only these two parantheses at $i$ and $i+1$, and let $d'$ be the ${\bf r}_d$-configuration consisting of all the symbols left after striking out the two parantheses at $i$ and $i+1$. Then, we have ${\bf r}_c=({\bf r}_d)_{d'}$.  

\end{remark}

We have now made the combinatorial preparations necessary to describe the ``multiplication table" of $K(\tilde H)$ in terms of the $u'_{\bf r}$'s. Both in the proposition and in the corollary following it, it is understood that we are working with $\tilde H=\tilde H(n)$, as usual; this is where the $n$ necessary in the definition of $n_{\bf r}$ comes from.

\begin{proposition}\prlabel{table}

For an $R=R(\tilde H)$-vector ${\bf r}$, the formula 
\begin{equation}\eqlabel{table}
f_{\bf r}=\sum_{c\in{\rm Conf}_{\bf r}}u'_{{\bf r}_c}
\end{equation}
holds in $K(\tilde H)$. 

\end{proposition}

\begin{proof}

This is more or less a tautology, once we translate the definition of $u'_{\bf r}$ given in \seref{1} using the notations employed here. Recall that we used the notation $u'_x$, $x\in A_R$ in \seref{1}, and then renamed that to $u'_{\bf r}$ by identifying the elements of the free monoid $A_R$ on $R$ with the $R$-vectors ${\bf r}$. The definition now reads 
\begin{equation}\eqlabel{def u' redux}
u'_{\bf r}=f_{\bf r}-\bigvee f_{{\bf r}_c},
\end{equation}
the supremum ranging over those ${\bf r}$-configurations $c$ with only two parantheses (necessarily, these would have to be a $($ at some position $i$, and its pair $)$ at position $i+1$). 

The proposition now follows by induction on the length of the vector ${\bf r}$, by applying the induction hypothesis to the vectors ${\bf r}_c$ and using the remark made above on the transitivity of configurations (\reref{trans}).  
\end{proof}

We also record the following consequence, as announced above:

\begin{corollary}\colabel{table}

The dimension of the comodule represented by $u'_{\bf r}$ is $n_{\bf r}$. 

\end{corollary}

\begin{proof}

With \prref{table} at our disposal, the proof is a simple counting argument plus induction by the length of ${\bf r}$, the base case of the induction ($|{\bf r}|=0,1$) being trivial. 

Fix ${\bf r}=(r_1,\ldots,r_k)$, and assume the statement is proven for shorter $R$-vectors. We then know that it holds for all ${\bf r}_c$, $c\in{\rm Conf}_{\bf r}$, except for $c=\emptyset$. Hence, by formula \equref{table} (and since $\dim (f_{\bf r})=n^k$), it suffices to show that
\begin{equation*}
n^{|{\bf r}|}=\sum_{c\in {\rm Conf}_{\bf r}} n_{{\bf r}_c}. 
\end{equation*}

To see how this comes about, remember that $n_{\bf r}$ is the cardinality of the set ${\mc O}_{\bf r}$, which is a certain collection of length $|{\bf r}|$ sequences with entries in $\ol{1,n}$; we will exhibit a bijection between the disjoint union of the sets ${\mc O}_{{\bf r}_c}$ and the set ${\ol 1,n}^{|{\bf r}|}$ of {\it all} such sequences. 

Fix an ${\bf r}$-configuration $c$, and consider the set ${\mc O}_{\bf r}^c$ of sequences in $\ol{1,n}^k$ defined by the following rules:

(a) if $i,i+1$ correspond to the empty symbol in $c$, then the same rules apply as for ${\mc O}_{\bf r}$, i.e. $(n_i,n_{i+1})\ne (n,n)$ if $r_{i+1}=r_i\pm 1$, $r_i$ even, and $(n_i,n_{i+1})\ne (1,1)$ if $r_{i+1}=r_i\pm 1$, $r_i$ odd;

(b) if $i<j$ hold parantheses $($ and respectively $)$ in $c$, then $n_i,n_j$ are both $n$ or both $1$, according to whether $r_i$ is even or odd, respectively. 

Given a sequence $n_1,\ldots,n_k$ in ${\mc O}_{\bf r}^c$, by simply deleting the $n_i$'s in the sequence for those $i$ which hold a paranthesis, we get a subsequence belonging to the set ${\mc O}_{{\bf r}_c}$. The opposite map from ${\mc O}_{{\bf r}_c}$ to ${\mc O}_{\bf r}^c$ is easily constructed by simply inserting the missing terms $n_i$ according to rule (b) above, so we have a bijection between the two sets. On the other hand, the set $\ol{1,n}^k$ of all length $k$ sequences with terms in the range $\ol{1,n}$ is clearly partitioned by the sets ${\mc O}_{\bf r}^c$, so we get the desired result. 
\end{proof}

We can now take care of part (b) of the theorem.

\renewcommand{\proofname}{Proof of \thref{maximal freeness} (b)}
\begin{proof}

``$\Leftarrow$" Suppose $H_1(F)$ is cosemisimple, and fix an $R=\mb Z/2$-vector ${\bf r}$. By \coref{table}, $u'_{\bf r}$ is a direct sum of simple comodules of total dimension $n_{\bf r}$. By \thref{freeness}, one of these comodules is $u_{\bf r}$. 

By the very definition of $u_{\bf r}$, the only matrix subcoalgebra of $X^{\bf r}$ which does not appear as a summand of $X^{\bf s}$ for some shorter vector $|{\bf s}|<|{\bf r}|$ is the one denoted above by $C_{\bf r}$, corresponding to the simple comodule $u_{\bf r}$. This means that $\dim(C_{\bf r})$ is precisely the number of irreducible monomials of type ${\bf r}$, i.e. $n_{\bf r}^2$ (see \reref{n_r}). But this then implies that the dimension of $u_{\bf r}$ is $n_{\bf r}$, so $u_{\bf r}$ accounts for the entire $u'_{\bf r}$.   

``$\Rightarrow$" We want to prove that if $u'_{\bf r}=u_{\bf r}$ for all $\mb Z/2$-vectors ${\bf r}$, then $\tilde H=H_1(F)$ is the sum of its matrix subcoalgebras $C_{\bf r}$ (corresponding respectively to the simple comodules $u_{\bf r}$).

Consider an element
\begin{equation}\eqlabel{grading}
x=\sum a^{\bf s}_{\bf ij} x^{\bf s}_{\bf ij}\in \tilde H
\end{equation}
in its standard form, where $a^{\bf s}_{\bf ij}$ are coefficients in the field $k$. If $t$ is a non-negative integer, denote
\[
x_t=\sum_{|{\bf s}|=t} a^{\bf s}_{\bf ij} x^{\bf s}_{\bf ij}.
\]
In other words, we are ``truncating" $x$ to its portion of length $t$. Typically, we will choose $t$ to be the top length of a monomial appearing in \equref{grading}.   

Now fix a $\mb Z/2$-vector ${\bf r}$. By hypothesis, $u'_{\bf r}=u_{\bf r}$ is simple; according to \coref{table}, its dimension is $n_{\bf r}$, so the dimension of its corresponding matrix coalgebra $C_{\bf r}$ is $n_{\bf r}^2$. But by \reref{n_r}, this is precisely the number of irreducible monomials of type ${\bf r}$. 

It has been noticed before that the map sending $x\in C_{\bf r}$ to $x_{|{\bf r}|}$ is an injection into the span of irreducible monomials of type ${\bf r}$. By the dimension count in the previous paragraph, $x\mapsto x_{|{\bf r}|}$ is an isomorphism of $C_{\bf r}$ onto this span. By induction on the length of the vectors, $x-x_{|{\bf r}|}$ is contained in the sum of all coalgebras $C_{{\bf s}}$, $|{\bf s}|<|{\bf r}|$, so finally, every irreducible monomial is contained in the sum of the subcoalgebras $C_{\bf r}$. 
\end{proof}
\renewcommand{\proofname}{Proof}

In the proof of \thref{maximal freeness} we will make use of known facts about the corepresentations of the quantized function algebra on $SL(2)$, which we denote here by $SL_q(2)$. As $SL_q(2)$ is one of the most well studied quantum groups, we do not recall the definition here; it can be found in numerous sources in the literature. The reference we will be making use of for the very basic results on its corepresentations that will actually come up here is \cite{KP}.  
Recall only that $q\in k^*$ is an invertible scalar. One usually considers it over fields $k$ of characteristic zero (typically $\mb C$), and furthermore, the corepresentations behave well (i.e. there is an isomorphism between the Grothendieck rings of $SL_q(2)$ and the usual $SL(2)$) when $q$ is not a root of unity. However, all the usual proofs go through in positive characteristic, even in the bad case when $q$ is a root of unity, as soon as its order is coprime to the characteristic; we invite the reader to check this as an exercise, going through the proofs in \cite{KP}, for example.  

$SL_q(2)$ has a fundamental $2\times 2$ matrix subcoalgebra denoted in \cite{KP} by 
\[
m=\begin{pmatrix}\alpha &\beta\\\gamma &\delta\end{pmatrix}
\]
which generates $SL_q(2)$ as an algebra. We also denote $m$ by $m_1$, and we use the same notation for the corresponding $2$-dimensional comodule, and its class in the Grothendieck ring; our $m$ is denoted by $u^{\frac 12}$ in \cite{KP}. One has, for small enough positive integers $t$, simple corepresentations $m_t$ which satisfy the Clebsch-Gordan multiplication table:
\begin{equation}\eqlabel{SL2}
m_t\otimes m\cong m_{t+1}\oplus m_{t-1}, 
\end{equation}
where $m_0$ stands for the trivial corepresentation. It follows that the dimension of $m_t$ is $t+1$. Here, $t$ less than half the order of $q$ minus 1 is ``small enough" in case $q$ is a root of unity. All of these corepresentations are self-dual. Only these partial results on the corepresentation theory of $SL_q(2)$ are important here; they follow immediately from the more detailed versions stated briefly at the end of \cite[Section 0]{KP} and proven in that paper.

\renewcommand{\proofname}{Proof of \thref{maximal freeness} (b)}
\begin{proof}

In the remarks immediately after \leref{not cosovereign'} we observed that it suffices to consider $\tilde H=H(2)$. Furthermore, by \leref{1-step'}, it suffices to prove the statement of the theorem for 1-step $R=\mb N$-vectors ${\bf r}$. 

Now fix a 1-step $\mb N$-vector ${\bf r}$. We know from \prref{table} that $f_{\bf r}$ can be broken up as the sum of all $u'_{{\bf r}_c}$'s, as $c$ ranges through all the ${\bf r}$-configurations. Moreover, \coref{table} says that the dimension of $u'_{{\bf r}}$ is $n_{{\bf r}}$. Since here the $n$ used in the calculation of $n_{{\bf r}_c}$ is $2$, it is a simple matter to compute $n_{\bf r}=|{\bf r}|+1$ (the fact that ${\bf r}$ is 1-step is crucial here). 

The plan of the proof is as follows:

Let $H$ be a Hopf algebra with a multiplicative matrix $m$ (we denote the corresponding $2$-dimensional comodule and its class in the Grothendieck ring by $m$ again). Let $\psi:H(2)\to H$ be the map sending $X^0$ to $m$, and denote the induced map on Grothendieck rings by the same symbol. If $\psi(f_{\bf r})$ contains some simple composition factor $m'$ of dimension $|{\bf r}|+1$ which does not appear as a composition factor in $\psi(f_{{\bf r}_c})$ for any non-empty ${\bf r}$-configuration $c$, then we must have 
\[
m'=\psi(u'_{\bf r}),
\]
and hence $u'_{\bf r}$ must be simple. 

Hence, it suffices to find $H,m$ as above, and this is where the $q$-analogues of $SL(2)$ come in. We take $H=SL_q(2)$ for some adequate $q$ (either not a root of unity, or, if the field $k$ is the algebraic closure of a finite field and we have no choice, a root of unity of order greater than $2|{\bf r}|+1$). $m$ will be the $m_1$ introduced above. Since $m$ is self-dual, it follows that $\psi(f_{\bf r})$ is precisely the $|{\bf r}|$'th tensor power of $m$. Finally, \equref{SL2} shows that $m'=m_{|{\bf r}|}$ has the desired properties. 
\end{proof}
\renewcommand{\proofname}{Proof}

We end by recasting the results obtained here in a form that is similar to \thref{Bichon}. The main observation is that given \thref{maximal freeness}, \prref{table} gives the formulas for the multiplication in the Grothendieck ring in terms of the basis $u'_{\bf r}=u_{\bf r}$ (in the cases covered by the theorem). In order to get explicit formulas (i.e. express the product $u_{\bf r}u_{\bf s}$ as a linear combination of the $u$'s), we need to introduce an operation on the monoid ring $\mb Z[A_R]$, similar to B\u anic\u a's $\odot$ mentioned in the introduction (\equref{Banica product}).

Recall that we defined an anti-endomorphism $*$ on the free monoid $A_R$ generated by $R$, given by sending the generator $\alpha_r$, $r\in R$ to $\alpha_{r+1}$, which extends by linearity to the monoid ring. In the discussion below, we identify words on $R$ (i.e. elements of $A_R$) with $R$-vectors ${\bf r}$ is the obvious way; the multiplication in the monoid $A_R$ is expressed in terms of vectors as concatenation (and written ${\bf rs}$ for vectors ${\bf r}$, ${\bf s}$), and the $*$ operation is given by
\[
(r_1, r_2,\ldots,r_k)^*=(r_k+1,\ldots,r_2+1,r_1+1).
\]

We call two vectors ${\bf r}, {\bf s}\in A_R$ {\it linked} and write ${\bf r} \sim {\bf s}$ if ${\bf r}^*={\bf s}$ or ${\bf s}^*={\bf r}$ (note that this is equivalent to ${\bf r}^*={\bf s}$ for $R=\mb Z/2$, which is the case treated in \cite{Ba}). Now consider the binary operation on $\mb Z[A_R]$, extended by linearity from the formula
\begin{equation}\eqlabel{Banica product'}
{\bf r}\odot {\bf s}=\sum {\bf ab},\ {\bf r}, {\bf s}\in A_R,
\end{equation}    
where the sum ranges over all possible ways of writing ${\bf r}={\bf at}$, ${\bf s}={\bf t'b}$ with ${\bf t} \sim {\bf t'}$. This operation is actually associative, and has the same unit as the usual multiplication in $\mb Z[A_R]$; all of this is easily checked.  

We extend the notation $u'_{\bf r}$ to $u'_a$ for any $a\in\mb Z[A_R]$ by linearity in $a$. In this setting, I claim that \prref{table} can be reformulated as follows:

\begin{prop}{4.6 bis}\prlabel{table'}

Let $\tilde H=H(n)$, $H_\infty(n)$, or $H_d(F)$. Then, the formula 
\begin{equation}\eqlabel{table'}
u'_{\bf r}u'_{\bf s}=u'_{{\bf r}\odot {\bf s}},\ \forall {\bf r}, {\bf s}\in A_R
\end{equation}
holds in the Grothendieck ring $K(\tilde H)$.

\end{prop}

\begin{proof}

This is proven by induction on $|{\bf r}|+|{\bf s}|$, the base case when ${\bf r}$ and ${\bf s}$ are both empty (i.e. of length zero) being trivial. Now fix ${\bf r}$, ${\bf s}$, and assume the statement is proven for smaller combined lengths of the two vectors. 

\prref{table} says that we have 
\begin{equation}\eqlabel{end1}
f_{\bf r}=\sum_{c\in {\rm Conf_{\bf r}}} u'_{{\bf r}_c},
\end{equation}
\begin{equation}\eqlabel{end2}
f_{\bf s}=\sum_{d\in {\rm Conf_{{\bf s}}}} u'_{{\bf s}_d},
\end{equation}
and
\begin{equation}\eqlabel{end3}
f_{{\bf rs}}=\sum_{e\in {\rm Conf_{{\bf rs}}}} u'_{({\bf rs})_e}.
\end{equation}

Since $f_{{\bf r}}f_{{\bf s}}=f_{{\bf rs}}$, we multiply \equref{end1} and \equref{end2} and compare the result to the right hand side of \equref{end3}. Apply the induction hypothesis to express all products $u'_{{\bf r}_c}u'_{{\bf s}_d}$ with $c\ne\emptyset$ or $d\ne\emptyset$ as a sum of $u'$ terms. This gives us some of the terms $u'_{({\bf rs})_e}$ in \equref{end3}, and the sum of the ones we do not get in this way will be exactly $u'_{{\bf r}}u'_{{\bf s}}$. 

It now remains to observe that the ${\bf rs}$-configurations $e$ which do not arise from products of the form $u'_{{\bf r}_c}u'_{{\bf s}_d}$ with $c,d$ not both empty are precisely those consisting of an unbroken string of `$($' symbols at the end of ${\bf r}$, followed by an unbroken string (necessarily of the same length) of `$)$' symbols at the beginning of ${\bf s}$. On the other hand, it's clear from our definitions that the $u'_{({\bf rs})_e}$ for such configurations $e$ are precisely the terms appearing in the definition \equref{Banica product'} of the product $\odot$.
\end{proof}

As promised above, we now have a complete, explicit description of the multiplication in $K=K(\tilde H)$ in the cases covered by \thref{maximal freeness}, when the $u'_{\bf r}$ form a basis for $K$ as a free abelian group: the multiplication table is described by \equref{table'}.



\end{document}